\documentclass[11pt]{amsart}

\usepackage{fullpage}
\usepackage{amsrefs}
\usepackage{amssymb}
\usepackage{color}

\newcommand{\Aut}{\operatorname{Aut}}
\newcommand{\Gbb}{\mathbb{G}}
\newcommand{\Gal}{\operatorname{Gal}}
\newcommand{\Z}{\mathbb{Z}}
\newcommand{\Dcal}{\mathcal{D}}
\newcommand{\Ical}{\mathcal{I}}
\newcommand{\Qcal}{\mathcal{Q}}

\newcommand{\U}{\mathrm{U}}
\newcommand{\Stab}{\operatorname{Stab}}
\newcommand{\smin}{\smallsetminus}

\newcommand{\UU}{\mathbb{U}}
\newcommand{\K}{\operatorname{K}^\mathrm{M}}

\newcommand{\trdeg}{\operatorname{trdeg}}
\renewcommand{\H}{\operatorname{H}}

\newcommand{\hotimes}{\widehat\otimes}
\newcommand{\Char}{\operatorname{Char}}
\renewcommand{\bar}{\overline}
\newcommand{\Zcal}{\mathcal{Z}}
\newcommand{\Gcal}{\mathcal{G}}
\newcommand{\Hom}{\operatorname{Hom}}
\newcommand{\Ccal}{\mathcal{C}}
\newcommand{\Sub}{\operatorname{Sub}}

\newcommand{\Frob}{\operatorname{Frob}}
\newcommand{\Q}{\mathbb{Q}}
\newcommand{\Ocal}{\mathcal{O}}
\newcommand{\mfrak}{\mathfrak{m}}
\newcommand{\Idiv}{\mathcal{I}^\mathrm{div}}
\newcommand{\Frac}{\operatorname{Frac}}
\newcommand{\cont}{\mathrm{cont}}
\newcommand{\Isom}{\operatorname{Isom}}
\newcommand{\Spec}{\operatorname{Spec}}
\newcommand{\divv}{\operatorname{div}}
\newcommand{\Pbb}{\mathbb{P}}

\newtheorem{maintheorem}{Theorem}
\newtheorem*{maintheorem*}{Main Theorem}

\newtheorem*{maincorollary*}{Corollary}

\newtheorem*{theorem*}{Theorem}

\numberwithin{equation}{section}
\newtheorem{proposition}[equation]{Proposition}
\newtheorem{theorem}[equation]{Theorem}
\newtheorem{lemma}[equation]{Lemma}
\newtheorem{keylemma}[equation]{Key Lemma}
\newtheorem{fact}[equation]{Fact}
\newtheorem{corollary}[equation]{Corollary}
\newtheorem*{claim*}{Claim}

\theoremstyle{definition}

\newtheorem*{problem*}{Problem}

\theoremstyle{remark}
\newtheorem{remark}[equation]{Remark}

\title{Recovering function fields from their integral $\ell$-adic cohomology with the Galois action}
\author{Adam Topaz}
\date{\today}
\thanks{This research was partially supported by NSERC and the University of Alberta.}
\address{Mathematical and Statistical Sciences, University of Alberta, 632 Central Academic Building, Edmonton AB T6G 2G1, Canada.}
\email{topaz@ualberta.ca}
\urladdr{https://adamtopaz.com/}

\begin{document}

\begin{abstract}
  In this note, we consider function fields of higher-dimensional algebraic varieties defined over non-local fields, and show how the Galois action on the cohomology such function fields can be used to parameterize their divisorial valuations.
  By applying a recent theorem of {\sc Pop}~\cite{1810.04768}, we observe that, in dimension $\geq 3$, this information is enough to completely determine the function field and the base-field in question.
\end{abstract}

\maketitle
\tableofcontents

\section{Introduction}\label{sec:introduction}

Fix a prime $\ell$.
For a field $F$ of characteristic $\neq \ell$, $\H^i(F,\Z_\ell(j))$ will denote the (continuous cochain) Galois cohomology of $F$ with coefficients in $\Z_\ell(j)$ (the $j$-th cyclotomic twist of $\Z_\ell$), and $F^i$ denotes the perfect closure of $F$.
One goal of the present note is to deduce the following ``anabelian'' result using $\ell$-adic Galois cohomology of geometric function fields endowed with the action of an absolute Galois group of a ``sufficiently global'' base-field $k_0$.

\begin{maintheorem}\label{maintheorem:anabeliantheorem-reconstruct}
  Suppose that $k_0$ is a perfect field of characteristic $\neq \ell$ which is not real-closed nor Henselian with respect to any non-trivial valuation.
  Let $K_0$ be a regular function field of transcendence degree $\geq 3$ over $k_0$.
  Let $k$ denote the algebraic closure of $k_0$ and put $K := K_0 \cdot k$.
  Then the fields $K^i$, $K_0^i$, $k$ and $k_0$, as well as the obvious inclusions among them, can be reconstructed (uniquely up-to Frobenius twists) from the following data:
  \begin{enumerate}
  \item The absolute Galois group $\Gal(k|k_0)$ of $k_0$, considered as a mere profinite group.
  \item The $\Z_\ell$-module $\H^1(K,\Z_\ell(1))$, and the action of the profinite group $\Gal(k|k_0)$ on the set ${\H^1(K,\Z_\ell(1))}_{/\Z_\ell^\times}$.
  \item The subset $\{ (x,y) \ : \ x,y \in \H^1(K,\Z_\ell(1)), \ x \cup y = 0\}$ of ${\H^1(K,\Z_\ell(1))}^2$.
  \end{enumerate}
\end{maintheorem}

Aside from the above theorem which shows that a function field can be ``reconstructed'' (meaning that there exists some \emph{algorithm} which starts with the given data and returns the fields in question as sets with addition and multiplication), we also prove a statement which is functorial with respect to isomorphisms of the data that appears there.
This is the content of the following theorem.

\begin{maintheorem}\label{maintheorem:anabeliantheorem-functorial}
  In the context of Theorem~\ref{maintheorem:anabeliantheorem-reconstruct}, suppose furthermore that $l_0$ is another field of characteristic $\neq \ell$ which is not real-closed nor Henselian with respect to any non-trivial valuation, and that $L_0|l_0$ is a regular function field of any transcendence degree.
  Let $l$ denote the algebraic closure of $l_0$ and put $L := L_0 \cdot l$.
  Suppose that
  \[ \phi : \H^1(K,\Z_\ell(1)) \cong \H^1(L,\Z_\ell(1)) \]
  is an isomorphism of complete $\Z_\ell$-modules and $\eta : \Gal(k|k_0) \cong \Gal(l|l_0)$ is an isomorphism of profinite groups, such that the following conditions hold:
  \begin{enumerate}
  \item The isomorphism $\phi$ is compatible with vanishing cup-products.
    Namely, for all elements $x,y \in \H^1(K,\Z_\ell(1))$, one has $x \cup y = 0$ in $\H^2(K,\Z_\ell(2))$ if and only if one has $\phi(x) \cup \phi(y) = 0$ in $\H^2(L,\Z_\ell(2))$.
  \item The isomorphism $\phi$ is compatible with the Galois action modulo $\Z_\ell^\times$.
    Namely, the induced bijection
    \[ \phi_{/\Z_\ell^\times} : {\H^1(K,\Z_\ell(1))}_{/\Z_\ell^\times} \cong {\H^1(L,\Z_\ell(1))}_{/\Z_\ell^\times} \]
    is equivariant with respect to the action of $\Gal(k|k_0)$ resp. $\Gal(l|l_0)$ via $\eta$.
  \end{enumerate}
  Then there exists an isomorphism of fields $\psi : K^i \cong L^i$ (unique up-to Frobenius twists) which restricts to an isomorphism $K_0^i \cong L_0^i$, and a (unique) $\epsilon \in \Z_\ell^\times$, such that $\epsilon \cdot \phi$ is the isomorphism induced by $\psi$, and such that $\eta$ is the isomorphism induced by $\psi$ via the identifications of Galois groups $\Gal(k|k_0) = \Gal(K^i|K_0^i)$ resp. $\Gal(l|l_0) = \Gal(L^i|L_0^i)$.
\end{maintheorem}

\begin{remark}
  The above theorems are perhaps best understood within the context of Bogomolov's program in birational anabelian geometry~\cite{MR1260938}, and the related Bogomolov-Pop conjecture~\cites{MR2421544,MR2891876}.
  In fact, the Bogomolov-Pop conjecture (with $\Z_\ell$-coefficients, cf.~\cite{MR3552242}) is equivalent to the assertion that the above two theorems hold true for an \emph{arbitrary} perfect field $k_0$ (resp. $l_0$).
  At this point, the Bogomolov-Pop conjecture is wide open in this level of generality; see~\cites{MR2421544,MR2859233,MR2891876} for the known cases.
Nevertheless, it seems reasonable to think that Theorems~\ref{maintheorem:anabeliantheorem-reconstruct} and~\ref{maintheorem:anabeliantheorem-functorial} would be sufficient for applications of this conjecture in most situations of interest in arithmetic geometry.
\end{remark}

\subsection{Strategy and Outline}
The technical core of this note studies the Galois action on the so-called \emph{quasi-prime divisors} of a given function field $K|k$, via their manifestation in cohomology.
Because of this, this paper includes a detailed overview of the \emph{local theory} from almost-abelian anabelian geometry, highlighting how it relates to the present work.

In the context of the theorems above, we will eventually be able to use the Galois action on such quasi-prime divisors to identify the \emph{divisorial valuations} of $K|k$ among all its quasi-divisorial valuations.
Along with some (now standard) cohomological calculations, this allows us to reduce the above two theorems to a recent result of {\sc Pop}~\cite{1810.04768}.
This is also precisely the reason why we must restrict to dimension $\geq 3$ (as this assumption appears in \emph{loc.\ cit.}), whereas we expect the above results to hold true in dimension $2$ as well.
Our parameterization of divisorial valuations using the Galois action holds true in arbitrary dimension $\geq 2$.

\subsection{Notation and Terminology}\label{subsection:notation-and-terminology}

Throughout the note $\ell$ will be a fixed prime, and $\Lambda$ will denote a quotient of $\Z_\ell$.
Given a $\Lambda$-module $M$, we write $\Sub(M)$ for the collection of $\Lambda$-submodules of $M$.
We use the notation $\hotimes$ to the denote the $\ell$-adically completed tensor product.
Hence, for a discrete abelian group $A$, one has $A \hotimes \Z/\ell^n = A/(\ell^n \cdot A)$ while $A \hotimes \Z_\ell$ denotes the $\ell$-adic completion of $A$.

For a field $F$ of characteristic $\neq \ell$, we will write $\Lambda(j)$ for the $j$-th cyclotomic twist of $\Lambda$ and consider the continuous cochain Galois cohomology groups $\H^i(F,\Lambda(j))$.
We have a canonical \emph{Kummer map}
\[ F^\times \rightarrow \H^1(F,\Lambda(1)), \]
which extends to a well-defined morphism of graded-commutative $\Lambda$-algebras
\begin{equation}\label{equation:norm-residue-morphism}
  \K_*(F) \hotimes \Lambda \rightarrow \H^*(F,\Lambda(*)).
\end{equation}
Here $\K_*(F)$ denotes the \emph{Milnor K-ring of $F$}, and $\H^*(F,\Lambda(*)) := \bigoplus_{n \geq 0} \H^n(F,\Lambda(n))$ is the usual cohomology ring of $F$ considered as a graded-commutative $\Lambda$-algebra.
The Bloch-Kato conjecture, which is now a highly celebrated theorem of {\sc Voevodsky-Rost} et al.~\cites{MR2811603,MR2597737,Rost1998} asserts that the map (\ref{equation:norm-residue-morphism}) above is an isomorphism of graded-commutative $\Lambda$-algebras.
This also implies that the cohomology groups considered above agree with the continuous \'etale cohomology groups in the sense of {\sc Jannsen} \cite{MR929536}.

Given a valued field $(F,v)$ such that $\Char(F) \neq \ell$, we will write $\Ocal_v$ for the valuation ring of $v$ with maximal ideal $\mfrak_v$, $\U_v$ for the $v$-units, and $\U_v^1 = 1 + \mfrak_v$ for the principal $v$-units.
Whenever $\Lambda$ as above is fixed, we will write $\UU_v$ for the (closure of the) image of the canonical map
\[ \U_v \hotimes \Lambda \rightarrow F^\times \hotimes \Lambda \xrightarrow{\text{Kummer}} \H^1(F,\Lambda(1)). \]
We will also write
\[ \partial_v : \H^1(F,\Lambda(1)) \rightarrow \Gamma_v \hotimes \Lambda \]
for the unique morphism through which the following map factors via the Kummer map:
\[ F^\times \xrightarrow{v} \Gamma_v \rightarrow \Gamma_v \hotimes \Lambda. \]
Note that $\UU_v$ is the kernel of $\partial_v$.
While the notation $\UU_v$ implicitly depends on the choice of $\Lambda$, this will always be clear from context and so should not cause any confusion.

Let $k_0$ be a perfect field with algebraic closure $k$.
Given a function field $K|k$, we say that $K|k$ is \emph{defined over $k_0$} provided that there exists a \emph{regular} function field $K_0|k_0$ such that $K = K_0 \cdot k$.
If we wish to specify the $K_0$ whose base-change is $K$ in the above sense, we will furthermore say that $K|k$ is \emph{defined over $k_0$ by $K_0$}.
Note that in this case we may identify $\Gal(k|k_0)$ with $\Gal(K|K_0) = \Gal(K^i|K_0^i)$ in the usual way.

Suppose that $K|k$ is defined over a perfect field $k_0$.
Then the $\Lambda$-algebras $\K_*(K)$ and $\H^*(K,\Lambda(*))$ obtain a canonical $\Lambda$-linear action of $\Gal(k|k_0)$, and the isomorphism
\[ \K_*(K) \hotimes \Lambda \cong \H^*(K,\Lambda(*)) \]
discussed above is $\Gal(k|k_0)$-equivariant.

We will say that a field $k_0$ is \emph{non-local} provided that $k_0$ is not real-closed and that $k_0$ has no non-trivial Henselian valuations.
Such non-local fields $k_0$ are ubiquitous throughout arithmetic geometry.
In fact, it seems that any field which is not obviously ``local'' in some way turns out to be non-local.
Here are some examples which are of particular interest, all of which are well-known to be non-local in the above sense.

\begin{fact}
The following fields are all non-local in the above sense.
\begin{enumerate}
\item Any algebraic extension of a finite field.
\item Any finitely-generated field.
\item The function field of any positive dimensional integral variety over any field.
\item Any purely inseparable extension of another non-local field.
\end{enumerate}
\end{fact}

\subsection*{Acknowledgements}
It is a pleasure to thank all who expressed interest in this work, and especially A. Cadoret and F. Pop.
The author warmly thanks T. Szamuely for his comments and suggestions regarding the related manuscript~\cite{TopazTorelli}, which eventually led to the developments in the present paper.

\section{Review of the local theory}\label{section:review-local-theory}

In this section, we review the so-called \emph{local theory} which allows us to detect valuations using cohomology.
The content in this section is not new, but is rather a distillation of the work done in \cites{MR1260938,MR1977585,MR3692015,MR3552293,TopazTorelli}.
For readers' sake, we aim to be as self-contained as possible in this section (see also the remarks concerning the local theory in \S\ref{section:concluding-remarks}).

\subsection{Basic Notation}
We let $\Lambda$ be a quotient of $\Z_\ell$, as above.
To greatly simplify the discussion, we will assume throughout this section that $\Lambda$ is a \emph{domain}, so that $\Lambda = \Z_\ell$ or $\Z/\ell$ (see~\cites{MR3692015,MR3552293} for the general case, which is much more technical).

For a field $F$, we write
\[ \Gcal(F,\Lambda) := \Hom(F^\times,\Lambda) \otimes_\Lambda \Frac\Lambda, \]
which we consider as a vector-space over $\Frac\Lambda$.
Note that the elements of $\Gcal(F,\Lambda)$ can be considered as homomorphisms $F^\times \rightarrow \Frac\Lambda$.
Note furthermore that one has
\[ \Hom(F^\times,\Lambda) = \Hom^{\cont}_\Lambda(F^\times \hotimes \Lambda,\Lambda) \]
and so, if $\Char F \neq \ell$, then $\Gcal(F,\Lambda)$ can be identified with
\[ \Hom^{\cont}_\Lambda(\H^1(F,\Lambda(1)),\Lambda) \otimes_\Lambda \Frac\Lambda \]
using Kummer theory.
In this case, we can also consider the elements of $\Gcal(F,\Lambda)$ as homomorphisms $\H^1(F,\Lambda(1)) \rightarrow \Frac\Lambda$.

We will also write $\Gcal^\pm(F,\Lambda)$ for the subspace of all elements $f \in \Gcal(F,\Lambda)$ such that $f(-1) = 0$.
Of course, since $\Lambda$ is assumed to be a domain, we have $\Gcal^\pm(F,\Lambda) = \Gcal(F,\Lambda)$ unless $\Lambda = \Z/2$, and, if $F$ contains $\mu_4$, then one always has $\Gcal^\pm(F,\Lambda) = \Gcal(F,\Lambda)$.

Given a valuation $v$ of $F$, we will write
\[ \Dcal_v = \{f \in \Gcal(F,\Lambda) \ : \ \forall x \in \U_v^1, \ f(x) = 0 \} \]
and similarly
\[ \Ical_v = \{f \in \Gcal(F,\Lambda) \ : \ \forall x \in \U_v, \ f(x) = 0 \} \]
Both are considered as subspaces of $\Gcal(F,\Lambda)$, with $\Ical_v \subset \Dcal_v$.
Note that one has a canonical restriction map $\Dcal_v \rightarrow \Gcal(Fv,\Lambda)$, whose kernel is $\Ical_v$.
Note also that $\Ical_v \subset \Gcal^\pm(F,\Lambda)$.

\subsection{Valuative Elements and Sets}\label{subsection:comparing-valuations}

Let $F$ be a field and let $\Sigma$ be a subset of $\Gcal(F,\Lambda)$.
We say that $\Sigma$ is \emph{valuative} provided that there exists some valuation $v$ of $F$ such that $\Sigma \subset \Ical_v$.
We say that an element $f \in \Gcal(F,\Lambda)$ is valuative provided that $\{f\}$ is valuative.

\begin{fact}\label{fact:valuative-morphisms}
  Let $\Sigma$ be a valuative subset of $\Gcal(F,\Lambda)$.
  Then there exists a unique coarsest valuation $v_\Sigma =: v$ of $F$ such that $\Sigma \subset \Ical_v$.
  Moreover, if $w$ is any valuation of $F$ such that $\Sigma \subset \Ical_w$, then $v_\Sigma$ is the coarsening of $w$ associated to the maximal convex subgroup of
  \[ w(\{x \in F^\times \ : \ \forall f \in \Sigma, \ f(x) = 0 \}). \]
\end{fact}
\begin{proof}
  Put
  \[ H := \{x \in F^\times \ : \ \forall f \in \Sigma, \ f(x) = 0 \}. \]
  Let $w$ be any valuation of $F$ such that $\Sigma \subset \Ical_w$, and let $v$  be the coarsening associated to the maximal convex subgroup of $w(H)$.
  We will show that $v$ doesn't depend on the choice of the original valuation $w$.
  In fact, it turns out that one has
  \[ \U_v = \{ x \in H \ : \ \forall y \in F^\times \smin H, \ (1+y) \in (x+y) \cdot H \}. \]

  The inclusion $\subseteq$ follows from the ultrametric inequality.
  Conversely, suppose that $x$ is contained in the set on the right.
  If $v(x) > 0$, then there exists some $y \in F^\times \smin H$ such that $0 < v(y) < v(x)$, for otherwise $v(H)$ would contain the interval $[0,v(x)]$, hence it would contain the convex subgroup generated by $v(x)$, thereby contradicting the defining property of $v$.
  But then $v(x+y) = v(y)$ while $1+y \in \U_v^1 \subset \U_v \subset H$.
  By definition, we have $1+y \in (x+y) \cdot H$ while $1+y \in H$, hence $(x+y) \in H$ and $y \in H$, since $(x+y) \cdot y^{-1} \in \U_v \subset H$.
  This is a contradiction, and the case $v(x) < 0$ provides a similar contradiction.
  In other words, it must be the case that $v(x) = 0$.
\end{proof}

\subsection{Alternating pairs}\label{subsection:alternating-pairs}

A pair of elements $f,g \in \Gcal(F,\Lambda)$ is called \emph{alternating} provided that for all $x \in F \smin \{0,1\}$, one has
\[ f(x) \cdot g(1-x) = f(1-x) \cdot g(x). \]
We say that a set $\Sigma \subset \Gcal(F,\Lambda)$ is alternating if it is pairwise alternating.
By the definition of Milnor K-theory, we see that $f,g$ form an alternating pair if and only if for all $x,y \in F^\times \hotimes \Lambda$ such that $\{x,y\} = 0$ in $\K_2(F) \hotimes \Lambda$, one has
\[ f(x) \cdot g(y) = f(y) \cdot g(x). \]
If $\Char F \neq \ell$, then we may identify $F^\times \hotimes \Lambda$ with $\H^1(F,\Lambda(1))$ via Kummer theory, and the Merkurjev-Suslin theorem~\cite{MR675529} shows that the alternating-pair condition can be alternatively tested on elements $x,y \in \H^1(F,\Lambda(1))$ whose cup-product is trivial in $\H^2(F,\Lambda(2))$.
We record this observation for later use in the following fact.

\begin{fact}\label{fact:cup-product-gives-alternating-pairs}
  Let $F$ be a field of characteristic $\neq \ell$.
  Suppose that $f,g \in \Gcal(F,\Lambda)$ are given, and consider them as morphisms
  \[ f,g : \H^1(F,\Lambda(1)) \rightarrow \Frac\Lambda \]
  via Kummer theory.
  Then $f,g$ form an alternating pair if and only if for all pairs of elements $x,y \in \H^1(F,\Lambda(1))$ such that $x \cup y = 0$ in $\H^2(F,\Lambda(2))$, one has $f(x) \cdot g(y) = f(y) \cdot g(x)$.
\end{fact}

\subsection{Connection with Valuations}

There is a deep connection between alternating pairs and valuations (this can already be seen in the classical calculations surrounding the construction of the tame symbol in Milnor K-theory).
We will record some of these basic facts here, all of which follow easily from the ultrametric inequality.

\begin{fact}\label{fact:residue-alternating-implies-alternating}
  Let $v$ be a valuation of $F$ and let $f,g \in \Dcal_v \cap \Gcal^\pm(F,\Lambda)$ be given.
  Then $f,g$ form an alternating pair if and only if their images in $\Gcal(Fv,\Lambda)$ form an alternating pair.
\end{fact}
\begin{proof}
  If $f,g$ form an alternating pair, then their restrictions in $\Gcal(Fv,\Lambda)$ clearly do as well.
  The converse is also straightforward, using the ultrametric inequality.
  We must show that the following equation
  \begin{equation}\label{equation:alternating-pair-condition}
    f(x) \cdot g(1-x) = f(1-x) \cdot g(x)
  \end{equation}
  holds for all $x \in F \smin \{0,1\}$.

  Indeed, if $v(x) = v(1-x) = 0$ then (\ref{equation:alternating-pair-condition}) holds by assumption.
  If $v(x) > 0$ resp. $v(1-x) > 0$ then $f(1-x) = g(1-x) = 0$ resp. $f(x) = g(x) = 0$, so equation (\ref{equation:alternating-pair-condition}) trivially holds.

  If $v(x) < 0$, then $f(1-x) = f(x)$ and $g(1-x) = g(x)$ since $x^{-1}-1 \in -1 \cdot  \U_v^1$ and $1-x = x \cdot (x^{-1} - 1)$.
  The same is also true if $v(1-x) < 0$ (replace $x$ with $1-x$).
  In any case, equation (\ref{equation:alternating-pair-condition}) holds true for all $x \in F \smin \{0,1\}$.
\end{proof}

\begin{lemma}\label{lemma:alternating-implies-comparable}
  Suppose that $\Sigma \cup \{g\}$ is an alternating subset of $\Gcal^\pm(F,\Lambda)$ and that all the elements of $\Sigma$ are valuative.
  Then $\Sigma$ is valuative and, letting $v$ denote the valuation associated to $\Sigma$, one has $g \in \Dcal_v$.
\end{lemma}
\begin{proof}
  First let us assume that $\Sigma = \{f\}$, and put $v = v_f$.
  It is straightforward to see that $g$ is trivial on $\U_{v}^1$.
  Indeed, if $v(x) > 0$, then $f(1-x) = 0$ hence
  \[ 0 = g(x) \cdot f(1-x) =  g(1-x) \cdot f(x). \]
  If $f(x) \neq 0$, it follows that $g(1-x) = 0$.
  If on the other hand $f(x) = 0$, then we may find some $y$ such that $0 < v(y) < v(x)$ and such that $f(y) \neq 0$ (otherwise $f$ is trivial on the interval $[0,v(x)]$, contradicting the definition of $v$).
  Since $v(x+y-xy) = v(y) > 0$, we have $f(y) = f(x+y-xy)$ and:
  \[ g((1-x)(1-y)) \cdot f(y) = g(x+y-xy) \cdot f(1-(x+y-xy)) = 0 \]
  hence $f(y) \neq 0$ implies that $g((1-x)(1-y)) = 0$.
  The argument above already shows that $g(1-y) = 0$, hence $g(1-x) = 0$ as well.

  Now remove the restriction on $\Sigma$, let $f_1,f_2 \in \Sigma$ be given, and put $v_i := v_{f_i}$.
  We will show that $v_1$ and $v_2$ are comparable.
  Using the argument above, we see that $f_1$ and $f_2$ are both trivial on $\U_{v_1}^1$ and $\U_{v_2}^1$.
  By the approximation theorem for valuations, we know that $\U_{v_1}^1 \cdot \U_{v_2}^1 = \U_w$ where $w$ is the finest common coarsening of $v_1$ and $v_2$.
  Thus, $f_1,f_2$ are both trivial on $\U_w$, hence $S := \{f_1,f_2\}$ is valuative, and $v_1$, $v_2$ are both coarsenings of $v_S$ by Fact~\ref{fact:valuative-morphisms}.
  This implies that $v_1$ and $v_2$ are comparable (and that $w$ must have been the coarser among the $v_1$, $v_2$).

  As argued above, we know that $g \in \Dcal_{v_f}$ for all $f \in \Sigma$.
  To conclude, let $v$ denote the valuative supremum of the $v_f$, which exists since all the $v_f$ are pairwise comparable.
  The arguments above show that $v = v_\Sigma$ and $g \in \Dcal_v$.
\end{proof}

\subsection{Detecting Valuations}\label{subsection:recovering-valuations}

We recall the following fundamental theorem relating alternating pairs to valuations of $F$.
We have distilled the essential parts of the arguments from~\cite{MR1977585}*{Proposition 4.1.2} and~\cite{MR3692015}*{\S11} in the proof below.
\begin{theorem}\label{theorem:fundamental-theorem-of-alternating-pairs}
  Let $f,g \in \Gcal^\pm(F,\Lambda)$ be given.
  Then the following are equivalent:
  \begin{enumerate}
  \item The pair $f,g$ is alternating.
  \item There exists some valuation $v$ of $F$ such that $f,g \in \Dcal_v$ and such that $\Ical_v \cap \langle f,g \rangle$ has codimension $\leq 1$ in $\langle f,g \rangle$.
  \end{enumerate}
\end{theorem}
\begin{proof}[Proof]
  The implication $(2) \Rightarrow (1)$ follows from Fact~\ref{fact:residue-alternating-implies-alternating}.
  For the non-trivial direction $(1) \Rightarrow (2)$, put $\Psi(-) = (f(-),g(-))$, considered as a homomorphism
  \[ \Psi : F^\times \rightarrow {(\Frac\Lambda)}^2. \]
  By the theory of Rigid Elements~\cite{MR910395}*{Theorem 2.16}, it is enough to show that, given any pair $x,y \in F \smin \{0,1\}$ such that $\Psi(1+z) \notin \{\Psi(1),\Psi(z)\}$ for $z = x,y$, the pair $\Psi(x)$, $\Psi(y)$ must be linearly dependent (see e.g.\ the argument from~\cite{TopazTorelli}*{Theorem A.9}).
  In the case where $\Lambda = \Z/2$, the theorem is taken care of directly using the theory of rigid elements (see~\cite{TopazTorelli}*{Theorem A.4}), so we will ignore this case below.

  The condition that $f,g$ are alternating tells us that for any $u,v \in F^\times$, the vector $\Psi(u-v) = \Psi(u) + \Psi(1-v/u)$ lies on an affine line containing $\Psi(u)$ and $\Psi(v)$, which is uniquely determined as soon as $\Psi(u) \neq \Psi(v)$.
  The same is true for $\Psi(u+v)$ since $\Psi(-1) = \Psi(1)$.

  Suppose that $x,y$ as above are given, and assume for a contradiction that $\Psi(x)$ and $\Psi(y)$ are linearly independent.
  Embed ${(\Frac\Lambda)}^2$ into $\Pbb := \Pbb^2(\Frac\Lambda)$ as the set of elements (in homogeneous coordinates) of the form $(1:a:b)$, $(a,b) \in {(\Frac\Lambda)}^2$, and compose with a $\Frac\Lambda$-projective-linear automorphism $\Sigma$ of $\Pbb$ to obtain a map
  \[ \Sigma \circ \Psi := \Phi : F^\times \rightarrow \Pbb \]
  which satisfies the following conditions:
  \begin{enumerate}
    \item For all $u,v \in K^\times$, $\Phi(u+v)$ and $\Phi(u-v)$ lie on a projective line containing $\Psi(u)$ and $\Psi(v)$.
    \item One has $\Phi(1) = (1:0:0)$, $\Phi(x) = (1:1:0)$, $\Phi(y) = (1:0:1)$, $\Phi(1+x) = (0:1:0)$ and $\Phi(1+y) = (0:0:1)$.
  \end{enumerate}

  A straightforward but tedious inductive argument using the conditions above shows that the image of $\Phi$ contains $(M:N:0)$ for all $(M,N) \in \Z^2$ with $M,N \geq 0$ relatively prime.\footnote{In fact, for coprime $(A,B,C) \in \Z^3 \smin \{(0,0,0)\}$, one has $\Phi(A+B \cdot x + C \cdot y) = (B+C-A:B:C)$.}
  See the argument in~\cite{TopazTorelli}*{Theorem A.9} for more detail, particularly  steps 1-4 in that proof.
  
  This will provide us with the required contradiction.
  Indeed, first recall that there exist $A,B \in \Frac\Lambda \smin \{0,1\}$ such that $\Psi(1+x) = A \cdot \Psi(x)$ and $\Psi(1+y) = B \cdot \Psi(y)$. 
  Also note that the automorphism $\Sigma$ sends the line at infinity $\{(0:u:v) \ : \ (u,v) \in (\Frac\Lambda)^2 \smin \{(0,0)\} \}$ in $\Pbb$ to the line between $(1:1-A:0)$ and $(1:0:1-B)$.

  In the case where $\Lambda = \Z/\ell$, we see that the image of $\Psi$ does not contain the line at infinity.
  However, as noted above, the image of $\Phi$ contains $(1:1-A:0)$.
  This is impossible.

  In the case where $\Lambda = \Z_\ell$, the definition of $\Gcal(F,\Lambda)$ ensures that the image of $f : F^\times \rightarrow \Q_\ell$ is contained in a set of the form $\ell^n \cdot \Z_\ell$ for some $n \in \Z$, and similarly for $g$.
  In other words, the image of $(f,g)$ is a lattice in $\Q_\ell^2$, and so the image of $\Psi$ cannot contain a sequence which converges $\ell$-adically to a point on the line at infinity.
  However, the closure of the image of $\Phi$ contains $(1:1-A:0)$ since $\Q_{\geq 0}$ is dense in $\Q_\ell$.
  Again, this is impossible.
\end{proof}

We further record the following result which follows easily from Theorem~\ref{theorem:fundamental-theorem-of-alternating-pairs} along with Lemma~\ref{lemma:alternating-implies-comparable}.
\begin{corollary}\label{corollary:fundamental-theorem-alternating-sequences}
  Let $\Dcal \subset \Gcal^\pm(F,\Lambda)$ be a subspace.
  Then the following are equivalent:
  \begin{enumerate}
  \item $\Dcal$ is alternating.
  \item There exists a valuation $v$ of $F$ such that $\Dcal \subset \Dcal_v$ and such that $\Dcal \cap \Ical_v$ has codimension $\leq 1$ in $\Dcal$.
  \end{enumerate}
  Furthermore, let $\Sigma$ denote the subset of $\Dcal$ consisting of all its valuative elements.
  If the above two equivalent conditions hold true, then the following hold true as well:
  \begin{enumerate}
  \item $\Sigma$ is a subspace of codimension $\leq 1$ in $\Dcal$.
  \item $\Sigma$ is valuative, and, putting $v := v_\Sigma$, one has $\Dcal \subset \Dcal_v$.
  \end{enumerate}
\end{corollary}

\section{Quasi-divisorial valuations}\label{section:quasi-divisorial-valuations}

Throughout this section and for the rest of the paper, we will restrict to the case where $K$ is a function field over an algebraically closed field $k$ of characteristic $\neq \ell$.

\subsection{Terminology}
A \emph{quasi-divisorial valuation} of $K|k$ is a valuation $v$ of $K$ which is minimal among valuations satisfying the following two conditions:
\begin{enumerate}
\item One has an isomorphism of abstract groups $vK/vk \cong \Z$.
\item One has $\trdeg(K|k) = \trdeg(Kv|kv) + 1$.
\end{enumerate}
In particular, a quasi-divisorial valuation $v$ has no transcendence defect in $K|k$, and hence $Kv$ is a function field of transcendence degree $\trdeg(K|k)-1$ over $kv$.
The collection of all quasi-divisorial valuations of $K|k$ will be denoted by $\Qcal(K|k)$.

Recall that $\UU_v$ denotes the closure of the image of $\U_v$ in $\H^1(K,\Lambda(1))$ under the Kummer map.
It turns out that a quasi-divisorial valuation $v$ is completely determined by $\UU_v$.
\begin{fact}\label{fact:qdvs-determined-by-UUv}
  Let $\Lambda$ be an arbitrary quotient of $\Z_\ell$.
  The canonical map
  \[ \Qcal(K|k) \rightarrow \Sub(\H^1(K,\Lambda(1))), \]
  given by $v \mapsto \UU_v$, is injective.
\end{fact}
\begin{proof}
  Given $v \in \Qcal(K|k)$, the induced map
  \[ \H^1(K,\Lambda(1)) \xrightarrow{\partial_v} \Gamma_v \hotimes \Lambda \cong \Lambda \twoheadrightarrow \Z/\ell \]
  is valuative.
  The minimality in the definition of $v$ ensures that $v$ is precisely the valuation associated to this valuative morphism.
  This valuation only depends on the kernel of this map which is $\UU_v + \ell \cdot \H^1(K,\Lambda(1))$.
  The assertion follows from Fact~\ref{fact:valuative-morphisms}.
\end{proof}

\subsection{Detecting Transcendence Degree}\label{subsection:recovering-transcendece-degree}

For the rest of this section, we again restrict to the case where $\Lambda$ is a domain, and we will use the notation introduced in \S\ref{section:review-local-theory}.
We will need the following characterization of the transcendence degree of $K|k$ using alternating pairs, which was originally observed by {\sc Bogomolov-Tschinkel}~\cite{MR1977585} and by {\sc Pop}~\cite{MR2735055}.

\begin{fact}\label{fact:recover-trdeg}
  In the above context, $\trdeg(K|k)$ is the maximal dimension of an alternating subspace of $\Gcal(K,\Lambda)$.
\end{fact}
\begin{proof}
  This is more-or-less a reformulation of Abhyankar's inequality for valuations of $K|k$.
  Indeed, let $\Dcal$ be an alternating subspace of $\Gcal(K,\Lambda)$.
  Then by Corollary~\ref{corollary:fundamental-theorem-alternating-sequences}, there is a valuation $v$ of $K$ such that $\Dcal \subset \Dcal_v$ and $\Ical_v \cap \Dcal$ has codimension $\leq 1$ in $\Dcal$, and such that $\Ical_v \cap \Dcal$ is the subset of $\Dcal$ consisting of its valuative elements.
  If $\Dcal \subset \Ical_v$, then we have
  \[ \dim(\Dcal) \leq \dim(\Ical_v) \leq \dim_\Q(vK/vk) \leq \trdeg(K|k). \]

  On the other hand, if $\Dcal \not \subset \Ical_v$, then there is a non-valuative element $g \in \Dcal$.
  It follows that $\Gcal(Kv,\Lambda)$ is non-trivial, for otherwise $\Dcal_v = \Ical_v$.
  Hence $\trdeg(Kv|kv) \geq 1$, and we have
  \[ \dim(\Dcal) = \dim(\Dcal \cap \Ical_v) + 1 \leq \dim_\Q(vK/vk) + \trdeg(Kv|kv) \leq \trdeg(K|k). \]

  Finally, if $X$ is a smooth $k$-variety with function field $K$ and $x \in X$ is a closed point, we may choose a regular sequence $t_1,\ldots,t_d$ at $x$, and use this sequence to construct a rank $d$ discrete valuation $v$ of $K$ (which is trivial on $k$).
  The value group $vK$ is isomorphic to $\Z^d$, and hence $\Ical_v$ has dimension $d$.
  The fact that $\Ical_v$ is alternating follows from Fact~\ref{fact:residue-alternating-implies-alternating}.
\end{proof}

\subsection{Parameterization of $\Qcal$}\label{subsection:parametarising-qcal}

We now give the promised parameterization of $\Qcal(K|k)$ using the $\Frac\Lambda$-module $\Gcal(K,\Lambda)$ and the collection of alternating pairs.
An alternative parameterization of quasi-prime divisors (as well as their compositions) was developed by {\sc Pop}~\cite{MR2735055}.
The parameterization we give below is somewhat simpler, and is adapted from the work of {\sc Bogomolov}~\cite{MR1260938} and {\sc Bogomolov-Tschinkel}~\cites{MR1977585,MR2421544}; these works implicitly assume that $k$ is the algebraic closure of a finite field, and hence all quasi-divisorial valuations of $K|k$ are trivial on $k$ (hence are \emph{divisorial}, see \S\ref{section:detecting-divisorial-valuations}).
We make no such restrictions on the base-field $k$, and so, at this point, we only obtain a parameterization of $\Qcal(K|k)$.

\begin{theorem}\label{theorem:qd-local-theory}
  Assume that $d := \trdeg(K|k) \geq 2$.
  Let $\Ical \subset \Gcal(K,\Lambda)$ be a 1-dimensional subspace.
  Then the following are equivalent:
  \begin{enumerate}
  \item There exists two $d$-dimensional subspaces $\Dcal_1,\Dcal_2 \subset \Gcal(K,\Lambda)$ such that $\Dcal_1 \cap \Dcal_2 = \Ical$ and such that $\Dcal_1$ and $\Dcal_2$ are both alternating.
  \item There exists a (unique) quasi-divisorial valuation $v$ of $K|k$ such that $\Ical_v = \Ical$.
  \end{enumerate}
  Furthermore, if these conditions hold true and $h \in \Ical_v$ is non-trivial, then one has
  \[ \Dcal_v = \{f \in \Gcal(K,\Lambda) \ : \ \text{$h,f$ is an alternating pair}\}. \]
\end{theorem}
\begin{proof}
  Suppose first that $v$ is a quasi-divisorial valuation of $K|k$.
  Let $w_1,w_2$ be two independent discrete rank $d-1$ valuations of $Kv|kv$ (see the argument in the last paragraph of the proof of Fact~\ref{fact:recover-trdeg}), and put $v_i := w_i \circ v$.
  Then $\Dcal_i := \Ical_{v_i}$ satisfies the properties required by (1).

  As for the converse, let $\Sigma_i$ denote the subset of valuative elements of $\Dcal_i$, which is then a valuative subspace of codimension $\leq 1$ in $\Dcal_i$ by Corollary~\ref{corollary:fundamental-theorem-alternating-sequences}.
  Let $v_i$ denote the valuation associated to $\Sigma_i$ and recall that $\Dcal_i \subset \Dcal_{v_i}$.

  We claim that $\Sigma_i = \Ical_{v_i}$.
  Indeed, we have $\Sigma_i \subset \Ical_{v_i}$ by definition, while $\Dcal_i \subset \Dcal_{v_i}$.
  If $\Sigma_i = \Dcal_i$, then $\dim(\Ical_{v_i}) \geq d$, and so Fact~\ref{fact:recover-trdeg} shows us that $\dim(\Ical_{v_i}) = d$ hence $\Sigma_i = \Ical_{v_i}$.
  If $\Sigma_i \neq \Dcal_i$, then $\Dcal_i$ contains a non-valuative element.
  The image of such an element in $\Gcal(Kv,\Lambda)$ is non-trivial, and hence $\trdeg(Kv_i|kv_i) \geq 1$.
  But this means that $\dim(\Ical_{v_i}) \leq \dim_\Q(v_i K/ v_i k) \leq d-1$.
  On the other hand, $\dim(\Sigma_i) \geq d-1$ so we see that $\Sigma_i = \Ical_{v_i}$.
  Note a similar argument also shows that $v_i$ has no transcendence defect in $K|k$.

  Let $h \in \Ical$ be non-trivial.
  We claim that $h$ (and hence $\Ical$) is valuative.
  If, for example, $v_1 \leq v_2$, then
  \[ \Sigma_1 = \Ical_{v_1} \subset \Ical_{v_2} = \Sigma_2 \]
  and hence $\Sigma_1 \subset \Ical = \Dcal_1 \cap \Dcal_2$.
  Since $\dim(\Sigma_1) \geq 1 = \dim(\Ical)$, this means that $\Sigma_1 = \Ical$, so that $h$ is indeed valuative.
  By symmetry, the same holds if $v_2 \leq v_1$.
  If, on the other hand, $v_1$ and $v_2$ are not comparable, then by the approximation theorem, $\U_{v_1}^1 \cdot \U_{v_2}^1 = \U_w$ where $w$ is the finest common coarsening of $v_1$ and $v_2$.
  But $h \in \Dcal_1 \cap \Dcal_2 \subset \Dcal_{v_1} \cap \Dcal_{v_2}$, hence it must be trivial on $\U_w$, which means it is an element of $\Ical_w$.
  In any case, this implies that $\Ical = \Sigma_1 \cap \Sigma_2$, since $\Sigma_i$ is the set of all valuative elements of $\Dcal_i$.

  Let $v := v_h$ denote the valuation of $K$ associated to the valuative element $h$.
  Note $v = v_\Ical$ and that $\Ical \subset \Sigma_i = \Ical_{v_i}$.
  By Fact~\ref{fact:valuative-morphisms}, we see that $v$ is a coarsening of $v_i$, and since $v_i$ is defectless the same is true for $v$.
  On the other hand, we have
  \[ \Ical_v \subset \Ical_{v_1} \cap \Ical_{v_2} = \Sigma_1 \cap \Sigma_2 = \Ical, \]
  while $\Ical \subset \Ical_v$ by the definition of $v$.
  Hence $\Ical_v = \Ical$.
  Finally, as $v$ is defectless, $vK/vk$ is isomorphic to $\Z^r$ for some $r$.
  But then $r = \dim(\Ical_v) = \dim(\Ical) = 1$, so that $vK/vk$ is isomorphic to $\Z$.
  As for the minimality of $v$, if $w$ is any coarsening of $v$ such that $wK/wk \cong \Z$, then $\Ical_v = \Ical_w$, so that $v = w$ by Fact~\ref{fact:valuative-morphisms}.
  Uniqueness follows from Fact~\ref{fact:valuative-morphisms}, while the assertion concerning $\Dcal_v$ follows from Fact~\ref{fact:residue-alternating-implies-alternating} and Lemma~\ref{lemma:alternating-implies-comparable}.
\end{proof}

\section{Detecting divisorial valuations}\label{section:detecting-divisorial-valuations}

A quasi-divisorial valuation $v$ of $K|k$ will be called \emph{divisorial} provided that $v$ is trivial on $k$.
It is known that for any divisorial valuation $v$ of $K|k$, there exists some normal model $X$ of $K|k$ and a codimension $1$ point $x \in X^{(1)}$ such that $\Ocal_v = \Ocal_{X,x}$; see~\cite{MR2316355}*{Remark/Definition 5.2, 5.3} for more on this.
We will write $\Dcal(K|k)$ for the collection of all divisorial valuations of $K|k$, which we consider as a subset of $\Qcal(K|k)$.

The goal of this section is to give a criterion which detects the divisorial valuations among the quasi-divisorial valuations using the Galois action of a non-local field.

\subsection{The canonical Henselian valuation of a field}

We will first sketch the argument for the following fact, which will play a crucial role later on.
\begin{fact}\label{fact:key-valtheory-fact}
  Let $F$ be a field, and let $L|F$ be a finite extension such that $L$ is not separably closed.
  Then the following are equivalent:
  \begin{enumerate}
  \item $F$ has a non-trivial Henselian valuation.
  \item $L$ has a non-trivial Henselian valuation.
  \end{enumerate}
\end{fact}

This key fact is certainly very well-known to valuation theorists, who can safely ignore this subsection.
Indeed, it follows from the properties of a well-known construction in valuation theory of the so-called \emph{canonical Henselian valuation} of a field.
This is a (possibly trivial) Henselian valuation which has good properties with respect to restriction in extensions.
Below we recall (without proof) this construction and the basic properties that imply the fact above.
For a general reference, refer the reader to \S4.4 of {\sc Engler-Prestel}~\cite{MR2183496}.

Let $F$ be an arbitrary field.
As always, valuations of $F$ are considered only up-to equivalence, so that we may identify the collection of valuations of $F$ with the collection of valuation subrings of $F$.
We will not distinguish between a valuation and its associated valuation ring.
Consider the following two collections of valuations of $F$:
\begin{enumerate}
\item First, $H_1(F)$ is the collection of all Henselian valuations of $F$ whose residue field is not separably closed.
\item Second, $H_2(F)$ is the collection of all Henselian valuations of $F$ whose residue field is separably closed.
\end{enumerate}
The following summarizes some basic facts concerning these two collections.

\begin{fact}[\cite{MR2183496}, Theorem 4.4.2]
  The following hold:
  \begin{enumerate}
  \item Any two valuations in $H_1(F)$ are comparable.
  \item If $H_2(F)$ is non-empty, then it contains a minimal element which is finer than all valuations in $H_1(F)$.
  \item If $H_2(F)$ is empty, then $H_1(F)$ contains a (unique) maximal element.
  \end{enumerate}
\end{fact}
The \emph{canonical Henselian valuation} of $F$ is defined to be the minimal element of $H_2(F)$ whenever $H_2(F)$ is non-empty, or the unique maximal element of $H_1(F)$ otherwise.
Note that the canonical Henselian valuation may in fact be the trivial valuation of $F$ (this is the case, for example, in our main case of interest where $F$ is non-local).

We write
\[ H(F) := H_1(F) \cup \{\text{the canonical Henselian valuation of $F$}\}.\]
The following fact summarizes the main properties of the canonical Henselian valuation of $F$ which we will need.

\begin{fact}[\cite{MR2183496}, Pg. 106--107]\label{fact:non-triv-canon-Henselian-valuations}
  Suppose that $L|F$ is a finite extension and that $L$ is not separably closed.
  Then the following hold:
  \begin{enumerate}
  \item $L$ admits a non-trivial Henselian valuation ring if and only if the canonical Henselian valuation of $L$ is non-trivial.
  \item Let $\Ocal$ be the canonical Henselian valuation ring of $L$.
    Then $\Ocal \cap F$ is a Henselian valuation ring of $K$ which is contained in $H(F)$.
  \end{enumerate}
\end{fact}

Our key Fact~\ref{fact:key-valtheory-fact} follows immediately from Fact~\ref{fact:non-triv-canon-Henselian-valuations}.
Indeed, if $F$ has a non-trivial Henselian valuation $v$ then the (unique) prolongation of $v$ to $L$ is again Henselian.
Conversely, by Fact~\ref{fact:non-triv-canon-Henselian-valuations}, if $L$ has a non-trivial Henselian valuation then the canonical Henselian valuation of $L$ is again non-trivial, and the restriction of this valuation to $F$ is again Henselian (and necessarily non-trivial).

\subsection{The Galois action on quasi-divisorial valuations}\label{subsection:galois-action-on-qdvs}

Assume that $K|k$ is defined over a perfect field $k_0$.
Recall this means that there exists a \emph{regular} function field $K_0|k_0$ such that $K = K_0 \cdot k$ where $k = \bar k_0$.
In this case, the canonical restriction maps
\[ \Gal(K^i|K_0^i) \xrightarrow{\cong} \Gal(K|K_0) \xrightarrow{\cong} \Gal(k|k_0) \]
are both isomorphisms.
We will tacitly identify these three Galois groups.

In particular, we see that $\Gal(k|k_0)$ acts on $K$ by automorphisms which restrict to automorphisms of $k$.
Hence, we obtain a canonical action of $\Gal(k|k_0)$ on $\Qcal(K|k)$.
This clearly restricts to an action on the subset $\Dcal(K|k)$, since this is just the subset of all $v \in \Qcal(K|k)$ such that $v|_k$ is trivial.

Furthermore, the canonical action of $\Gal(k|k_0)$ on $\H^*(K,\Lambda(*))$ arises from this action on $K$ (and the cyclotomic twist of the coefficients).
In particular, the Kummer isomorphism
\[ K^\times \hotimes \Lambda \cong \H^1(K,\Lambda(1)) \]
is $\Gal(k|k_0)$ equivariant, where the action is via the identification with $\Gal(K|K_0)$ on the left.
We also obtain a corresponding action of $\Gal(k|k_0)$ on $\Sub(\H^1(K,\Lambda(1)))$.

We record the following further compatibility, which is now obvious given the definition of $\UU_v$ for $v \in \Qcal(K|k)$, since it will be used later.

\begin{fact}\label{fact:qcal-inclusion-galois-equivariant}
  Suppose that $K|k$ is defined over $k_0$.
  The canonical injective map
  \[ \Qcal(K|k) \hookrightarrow \Sub(\H^1(K,\Lambda(1))), \ \ v \mapsto \UU_v \]
  is equivariant with respect to the action of $\Gal(k|k_0)$.
\end{fact}

\subsection{Detecting divisorial valuations}\label{subsection:detect-divisorial-valuations}

We are finally prepared to state and prove our Key Lemma, which distinguishes the divisorial valuations among the quasi-divisorial ones, using the Galois action.

\begin{keylemma}\label{keylemma:galois-invariant-qdvs-are-dvs}
  Assume that $k_0$ is a perfect non-local field and that $K|k$ is defined over $k_0$.
  Then one has
  \[ \Dcal(K|k) = \bigcup_{k_1|k_0} {\Qcal(K|k)}^{\Gal(k|k_1)}\]
  where $k_1$ varies over the finite extensions of $k_0$.
  In particular, a quasi-divisorial valuation $v$ of $K|k$ is divisorial if and only if
  \[ \{ \sigma \in \Gal(k|k_0) \ : \ \sigma(\UU_v) = \UU_v \} \]
  is an open subgroup of $\Gal(k|k_0)$.
\end{keylemma}
\begin{proof}
  The second assertion clearly follows from the first, using Fact~\ref{fact:qcal-inclusion-galois-equivariant}.
  We shall prove the first assertion.
  The inclusion $\subseteq$ is clear.
  Indeed, if $v$ is a divisorial valuation, then it is defined by some regular point of codimension $1$ on some normal model $X$ of $K|k$, which is in turn defined over some finite extension $k_1$ over $k_0$.
  One clearly has, in this case, $\Gal(k|k_1) \subset \Stab_{\Gal(k|k_0)}(v)$.

  For the converse, assume that $v \in \Qcal(K|k)$ is fixed by $\Gal(k|k_1)$.
  Let $w$ denote the restriction of $v$ to $k$, $w_1$ the restriction of $w$ to $k_1$, and $w_0$ the restriction of $w$ to $k_0$.
  For all $\sigma \in \Gal(k|k_1)$, one has
  \[ w = v|_k = (\sigma v)|_k = \sigma(v|_k) = \sigma w. \]
  Hence $\Gal(k|k_1)$ is contained in the decomposition group $\Zcal_{w|w_0}$.
  Namely, $w_1$ is Henselian.

  Since $k_0$ is not real-closed by assumption, it follows that $k_1$ is not separably closed.
  Finally, we note that $w_1$ must be trivial, for otherwise $k_0$ is Henselian with respect to a non-trivial valuation by Fact~\ref{fact:non-triv-canon-Henselian-valuations}.
  In other words, $w$ must be trivial, so that $v$ is divisorial.
\end{proof}

\section{Recovering function fields}\label{section:recovering-function-fields}

In this section we conclude the proof of Theorems~\ref{maintheorem:anabeliantheorem-reconstruct} and~\ref{maintheorem:anabeliantheorem-functorial} by reducing to the main theorem of~\cite{1810.04768}.
As \emph{loc.\ cit.} uses \emph{pro-$\ell$ abelian-by-central} Galois groups as its input, we must first discuss how to recover these Galois groups from Galois cohomology.
This translation is routine and has appeared (in some incarnation) several times before in the context of anabelian geometry; for more on this, see~\cite{MR2885583},~\cite{MR3552242}*{\S8} and/or~\cite{MR3827205}*{\S4} for an analogous context which deals with finite coefficients (which is more technical due to the presence of the Bockstein morphism).
Throughout this section, we restrict our attention to the case where $\Lambda = \Z_\ell$.

\subsection{Translation to Galois groups}\label{subsection:translation-to-galois-groups}

For a pro-$\ell$ group $\Pi$, we write $\Pi^{(n)}$ for the central descending series of $\Pi$.
This series is defined as follows:
\[ \Pi^{(1)} = \Pi, \ \ \Pi^{(n+1)} = \overline{[\Pi,\Pi^{(n)}]}. \]
We put $\Pi^a := \Pi/\Pi^{(2)}$ resp. $\Pi^c := \Pi/\Pi^{(3)}$ and refer to them as the maximal \emph{abelian} resp. \emph{abelian-by-central} quotients of $\Pi$.

We will assume that $\Pi^a$ is torsion-free for the rest of this section.
A \emph{minimal free pro-$\ell$ presentation} of $\Pi$ is a morphism of pro-$\ell$ groups
\[ \pi : S \rightarrow \Pi \]
where $S$ is a free pro-$\ell$ group, such that the induced map
\[ \pi^a : S^a \rightarrow \Pi^a \]
is an isomorphism (this forces $\pi$ to be surjective).
Such minimal free pro-$\ell$ presentations always exist (under our torsion-freeness assumption), and for such a presentation it follows that $S^{(2)}$ contains the kernel of $\pi$.

Our goal is to explicitly describe the kernel of the induced (surjective) map
\[ \pi^c : S^c \twoheadrightarrow \Pi^c \]
using the inflation map in cohomology.

First, note that one has an extension of profinite groups
\[ 1 \rightarrow S^{(2)}/S^{(3)} \rightarrow S/S^{(3)} \rightarrow \Pi^a \rightarrow 1. \]
Given any morphism $S^{(2)}/S^{(3)} \rightarrow \Lambda$, we may push out the above extension to obtain a corresponding extension of $\Pi^a$ by $\Lambda$ considered up-to equivalence of extensions, which we may in turn identify with an element of $\H^2(\Pi^a,\Lambda)$.
To summarize, we obtain a map (which is canonical once $\pi$ is fixed):
\[ \Hom^{\cont}(S^{(2)}/S^{(3)}, \Lambda ) \xrightarrow{d_2} \H^2(\Pi^a,\Lambda). \]
Since $S$ is free, it is easy to see that this map is in fact an \emph{isomorphism}.

\begin{lemma}\label{lemma:cohom-to-galois-reconstruct}
  Assume that $\Pi^a$ and $\Pi^c$ are torsion-free, and let $\pi : S \twoheadrightarrow \Pi$ be a minimal free pro-$\ell$ presentation.
  Let $\Ccal$ denote the kernel of the canonical map $\H^2(\Pi^a,\Lambda) \rightarrow \H^2(\Pi,\Lambda)$ and put
  \[ T = \bigcap_{f \in d_2^{-1}(\Ccal)} \ker(f),\]
  considered as a subgroup of $S^{(2)}/S^{(3)}$.
  Then $T$ is a normal subgroup of $S^c$, and $\pi$ induces an isomorphism $S^c/T \cong \Pi^c$ which lifts the isomorphism $S^a \cong \Pi^a$ induced by $\pi$.
\end{lemma}
\begin{proof}
  Using our torsion-freeness assumptions, this boils down to the assertion that the kernel $\Ccal$ of
  \[ \H^2(\Pi^a,\Lambda) \rightarrow \H^2(\Pi,\Lambda) \]
  agrees with the kernel, say $\Ccal^c$, of
  \[ \H^2(\Pi^a,\Lambda) \rightarrow \H^2(\Pi^c,\Lambda). \]
  Clearly, $\Ccal^c$ is contained in $\Ccal$.
  The converse is easy to see by the definition of $\Pi^c$ (being the maximal pro-$\ell$ two-step nilpotent quotient of $\Pi$), along with the fact that any extension of $\Pi^a$ by $\Lambda$ is pro-$\ell$ and two-step nilpotent.
\end{proof}

\begin{remark}\label{remark:construct-from-cohom}
  Lemma~\ref{lemma:cohom-to-galois-reconstruct} can be used to construct a pro-$\ell$ group which is abstractly isomorphic to $\Pi^c$ using the following data:
  \begin{enumerate}
    \item[A.] The pro-$\ell$ group $\Pi^a$.
    \item[B.] the kernel $\Ccal$ of $\H^2(\Pi^a,\Lambda) \rightarrow \H^2(\Pi,\Lambda)$.
  \end{enumerate}
  This is accomplished with the following steps:
  \begin{enumerate}
    \item Choose a minimal free pro-$\ell$ presentation $\pi : S \rightarrow \Pi^a$ and consider the isomorphism
      \[ \Hom^{\cont}(S^{(2)}/S^{(3)}, \Lambda ) \xrightarrow{d_2} \H^2(\Pi^a,\Lambda). \]
    \item Write $T$ for the intersection
      \[ T := \bigcap_{f \in d_2^{-1}(\Ccal)} \ker(f).\]
  \end{enumerate}
  Then Lemma~\ref{lemma:cohom-to-galois-reconstruct} ensures that $T$ is normal in $S^c$ and $S^c/T$ is isomorphic to $\Pi^c$.
\end{remark}

We will also need the following functorial variant of the above fact.

\begin{lemma}\label{lemma:cohom-to-galois-functorial}
  Let $\Pi_1,\Pi_2$ be two profinite groups such that $\Pi_i^a$ and $\Pi_i^c$ are torsion-free for $i = 1,2$.
  Let $f : \Pi_1^a \rightarrow \Pi_2^a$ be an isomorphism, and let $\Ccal_i$ denote the kernel of
  \[ \H^2(\Pi_i^a,\Lambda) \rightarrow \H^2(\Pi_i,\Lambda). \]
  Then the following are equivalent:
  \begin{enumerate}
  \item The isomorphism $f$ lifts to an isomorphism $\Pi_1^c \cong \Pi_2^c$.
  \item The isomorphism $f^* : \H^2(\Pi_2^a,\Lambda) \cong \H^2(\Pi_1^a,\Lambda)$ restricts to an isomorphism $\Ccal_2 \cong \Ccal_1$.
  \end{enumerate}
\end{lemma}

\begin{remark}
  Let $F$ be any field of characteristic $\neq \ell$ which contains $\mu_{\ell^\infty}$, and let $\Pi_F$ denote its maximal pro-$\ell$ Galois group.
  Then both $\Pi_F^a$ and $\Pi_F^c$ are torsion-free.
  Indeed, Kummer theory provides an identification of $\Pi_F^a$ with $\Hom(F^\times,\Z_\ell(1))$ which is torsion-free since $\Z_\ell(1) \cong \Z_\ell$, while the kernel of $\Pi_F^c \rightarrow \Pi_F^a$ can be identified with
  \[ {\Hom(L^\times,\Z_\ell(1))}^{\Pi_F^a}\]
  where $L$ is the Galois extension of $F$ such that $\Gal(L|F) = \Pi_K^a$.
\end{remark}

\subsection{Divisorial Inertia Elements}\label{subsection:divisorial-inertia-elements}

We now review the result of {\sc Pop}~\cite{1810.04768} alluded to above.
Let $K$ be a function field over an algebraically closed field $k$ of characteristic $\neq \ell$, and let $\Pi_K$ denote the maximal pro-$\ell$ Galois group of $K$.
An element $\gamma$ of $\Pi_K^a$ is called a \emph{divisorial inertia element} provided that there exists some \emph{divisorial valuation} $v$ of $K|k$ such that $\gamma$ is contained in the inertia subgroup of $\Pi_K^a$ associated to $v$ (since $\Pi_K^a$ is abelian, the inertia group does not depend on a choice of prolongation of $v$).
The collection of all divisorial inertia elements of $\Pi_K^a$ will be denoted by $\Idiv_K$.

Now suppose that $L$ is another function field over an algebraically closed field $l$.
We will write
\[ \Isom^i(K,L) := \Isom(K^i,L^i) \]
and ${\Isom^i(K,L)}_{/\langle \Frob \rangle}$ for the quotient of $\Isom^i(K,L)$ by the action of $\Frob_{L^i}$ (acting by composition).

Consider $\Isom(\Pi_L^a,\Pi_K^a)$ and note that there is a canonical action of $\Z_\ell^\times$ on it by left multiplication.
We write ${\Isom(\Pi_L^a,\Pi_K^a)}_{/\Z_\ell^\times}$ for the quotient by this action.
Note that any element of ${\Isom(\Pi_L^a,\Pi_K^a)}_{/\Z_\ell^\times}$ with representative $\phi : \Pi_L^a \cong \Pi_K^a$, induces a canonical bijection
\[ \phi_{/\Z_\ell^\times} : {(\Pi_L^a)}_{/\Z_\ell^\times} \cong {(\Pi_K^a)}_{/\Z_\ell^\times}\]
which is independent of the choice of representative $\phi$.

We will write ${\Isom^c(\Pi_L^a,\Pi_K^a)}_{/\Z_\ell^\times}$ for the subset of ${\Isom(\Pi_L^a,\Pi_K^a)}_{/\Z_\ell^\times}$ consisting of those elements which are represented by an isomorphism $\phi : \Pi_L^a \cong \Pi_K^a$ such that $\phi$ lifts to an isomorphism of pro-$\ell$ groups $\Pi_L^c \cong \Pi_K^c$.
Similarly, we will write ${\Isom_{\divv}(\Pi_L^a,\Pi_K^a)}_{/\Z_\ell^\times}$ for the subset of ${\Isom(\Pi_L^a,\Pi_K^a)}_{/\Z_\ell^\times}$ consisting of elements which are represented by an isomorphism $\phi : \Pi_L^a \cong \Pi_K^a$ which induces a bijection $\Idiv_L \cong \Idiv_K$.
Finally, we will write ${\Isom^c_{\divv}(\Pi_L^a,\Pi_K^a)}_{/\Z_\ell^\times}$ for the intersection of ${\Isom_{\divv}(\Pi_L^a,\Pi_K^a)}_{/\Z_\ell^\times}$ and ${\Isom^c(\Pi_L^a,\Pi_K^a)}_{/\Z_\ell^\times}$, both considered as subsets of ${\Isom(\Pi_L^a,\Pi_K^a)}_{/\Z_\ell^\times}$.

Note that any isomorphism $\phi : K^i \cong L^i$ of fields induces an isomorphism $\Pi_L^a \cong \Pi_K^a$.
In other words, we obtain a canonical map
\[ \Isom^i(K,L)  \rightarrow {\Isom(\Pi_L^a,\Pi_K^a)}_{/\Z_\ell^\times}. \]
which factors through ${\Isom^i(K,L)}_{/\langle \Frob \rangle}$.
Clearly, the image of this map lands in the subset
\[ {\Isom^c_{\divv}(\Pi_L^a,\Pi_K^a)}_{/\Z_\ell^\times}. \]
The main result of \emph{loc.\ cit.} can now be summarized in the following theorem.
\begin{theorem}[{\sc Pop}~\cite{1810.04768}]\label{theorem:Pop-recover-functionfields-from-divisorial-inertia}
  Suppose that $K$ is a function field over an algebraically closed field $k$ of characteristic $\neq \ell$, and that $\trdeg(K|k) \geq 3$.
  Then there is a group-theoretical recipe which reconstructs $K^i|k$ (uniquely up-to Frobenius twists) from the following data:
  \begin{enumerate}
    \item $\Pi_K^c$ considered as a pro-$\ell$ group.
    \item The subset $\Idiv_K$ of divisorial inertia elements in $\Pi_K^a$.
  \end{enumerate}
  This recipe is functorial with respect to isomorphisms in the following sense.
  Suppose that $L|l$ is another function field over another algebraically closed field.
  Then the canonical map
  \[ {\Isom^i(K,L)}_{/\langle \Frob \rangle} \rightarrow {\Isom^c_{\divv}(\Pi^a_L,\Pi_K^a)}_{/\Z_\ell^\times} \]
  is a bijection.
\end{theorem}

\subsection{Proof of the Main Theorems}\label{subsection:proofs-of-main-theorems}
In this subsection we complete the proofs of Theorems~\ref{maintheorem:anabeliantheorem-reconstruct} and~\ref{maintheorem:anabeliantheorem-functorial}.

\begin{proof}[Proof of Theorem~\ref{maintheorem:anabeliantheorem-reconstruct}]
  We are given the following data:
  \begin{enumerate}
    \item[A.] The profinite group $\Gamma := \Gal(k|k_0)$.
    \item[B.] The $\Z_\ell$-module $\H^1(K,\Z_\ell(1))$ and the canonical action of $\Gamma$ on ${\H^1(K,\Z_\ell(1))}_{/\Z_\ell^\times}$.
    \item[C.] The set $\{ (x,y) \ : \ x,y \in \H^1(K,\Z_\ell(1)), \ x \cup y = 0 \}$.
  \end{enumerate}
  Using this data, it is now a simple matter of putting together the various constructions discussed above to conclude the proof.
  To do this, we follow the following steps:
  \begin{enumerate}
    \item Construct $\Sub(\H^1(K,\Z_\ell(1)))$ using $\H^1(K,\Z_\ell(1))$, as well as the natural action of $\Gamma$ on $\Sub(\H^1(K,\Z_\ell(1)))$ which is determined by the (given) action of $\Gamma$ on ${\H^1(K,\Z_\ell(1))}_{/\Z_\ell^\times}$.
      Also construct the following:
    \begin{enumerate}
      \item $\Gcal(K,\Z_\ell)$, identified as $\Hom^{\cont}_\Lambda(\H^1(K,\Z_\ell(1)),\Z_\ell) \otimes_{\Z_\ell} \Q_\ell$.
      \item The canonical pairing
        \[ \H^1(K,\Z_\ell(1)) \times \Gcal(K,\Z_\ell) \rightarrow \Q_\ell \]
        obtained from our construction of $\Gcal(K,\Z_\ell)$.
      \item The collection of all alternating pairs of elements of $\Gcal(K,\Z_\ell)$, using Fact~\ref{fact:cup-product-gives-alternating-pairs}.
    \end{enumerate}
  \item Parameterize $\Qcal(K|k)$ using Fact~\ref{fact:recover-trdeg} and Theorem~\ref{theorem:qd-local-theory} as the collection of subspaces of $\Gcal(K,\Z_\ell)$ of the form $\Ical_v$, $v \in \Qcal(K|k)$.
    Note that $\UU_v$ is the orthogonal of $\Ical_v$ with respect to the pairing above, so this allows us to construct the image of the injective map
    \[ \Qcal(K|k) \hookrightarrow \Sub(\H^1(K,\Z_\ell(1))), \ v \mapsto \UU_v. \]
  \item Using the $\Gamma$ action on $\Sub(\H^1(K,\Z_\ell(1)))$, construct the image of
    \[ \Dcal(K|k) \hookrightarrow \Sub(\H^1(K,\Z_\ell(1))) \]
    which consists of the elements in the image of
    \[ \Qcal(K|k) \hookrightarrow \Sub(\H^1(K,\Z_\ell(1))) \]
    whose $\Gamma$-stabilizer is open in $\Gamma$.
    This is Key Lemma~\ref{keylemma:galois-invariant-qdvs-are-dvs}.
  \item Construct $\Pi_K^a$ and the canonical action of $\Gamma$ on ${(\Pi_K^a)}_{/\Z_\ell^\times}$.
    Here we identify $\Pi_K^a$ with the $\Z_\ell$-dual of $\H^1(K,\Z_\ell(1))$.
    As we are only concerned with the action of $\Gamma$ on ${(\Pi_K^a)}_{/\Z_\ell^\times}$, the cyclotomic twist in the coefficients is irrelevant.
  \item Identify $\H^2(\Pi_K^a,\Z_\ell)$ with the $\ell$-adic completion of $\wedge^2 \H^1(K,\Z_\ell(1))$ (via (4) above), and use the set $\Ccal := \{(x,y) \ : \ x,y \in \H^1(K,\Z_\ell(1)), \ x \cup y = 0\}$, along with the Merkurjev-Suslin theorem~\cite{MR675529}, to construct the kernel of the inflation map
    \[ \H^2(\Pi_K^a,\Z_\ell) \rightarrow \H^2(\Pi_K,\Z_\ell) \]
    as the closure of the subgroup generated by elements of the form $x \wedge y$ for $(x,y) \in \Ccal$.
    Here we identify $\H^2(\Pi_K,\Z_\ell)$ with its image in $\H^2(K,\Z_\ell)$ via the inflation map (which is injective since $K(\ell)$ is $\ell$-closed), and we identify $\H^i(K,\Z_\ell)$ with $\H^i(K,\Z_\ell(j))$ via an isomorphism of $\Gal_K$-modules $\Z_\ell \cong \Z_\ell(1)$.
  \item Use step (5) along with Lemma~\ref{lemma:cohom-to-galois-reconstruct} to construct $\Pi_K^c$ as an abstract pro-$\ell$ group along with the surjective map $\Pi_K^c \twoheadrightarrow \Pi_K^a$, where $\Pi_K^a$ was constructed as above (see Remark~\ref{remark:construct-from-cohom}).
  \item Consider the Kummer pairing
    \[ \Pi_K^a \times \H^1(K,\Z_\ell(1)) \rightarrow \Z_\ell(1) \cong \Z_\ell. \]
    Given $v \in \Dcal(K|k)$, the orthogonal of $\UU_v$ under this pairing is the inertia group of $v$ in $\Pi_K^a$.
    Hence using (3) we may construct the subset $\Idiv_K$ of divisorial inertia elements inside of $\Pi_K^a$.
    The image of this subset in ${(\Pi_K^a)}_{/\Z_\ell^\times}$ is visibly stable with respect to the action of $\Gamma$.
  \item Using Pop's Theorem (Theorem~\ref{theorem:Pop-recover-functionfields-from-divisorial-inertia}), we may reconstruct $K^i$ uniquely up-to Frobenius twists.
    The $\Gamma$ equivariant nature of the above constructions further provide us with the image of
    \[ \Gamma \cong \Gal(K^i|K_0^i) \rightarrow \Aut(K^i)/\langle \Frob_{K^i} \rangle. \]
    The intersection of $\Gal(K^i|K_0^i)$ with $\langle \Frob_{K^i} \rangle$ in $\Aut(K^i)$ is trivial, hence this map is injective.
    In any case, once we fix a copy of $K^i$ which we reconstruct using the above data (which involves choices) we also obtain the canonical action of $\Gamma$ on this copy of $K^i$.
  \item The maximal algebraically closed subfield of $K^i$ is $k$, hence we obtain $K^i$ and its subfield $k$, while taking $\Gamma$-invariants in $K^i$ resp. $k$ yields $K_0^i$ resp. $k_0$.
    The obvious inclusions among $K^i$, $K_0^i$, $k$ and $k_0$ are visibly determined by this construction.
  \end{enumerate}
  This concludes the proof of Theorem~\ref{maintheorem:anabeliantheorem-reconstruct}.
\end{proof}

Before we proceed to prove Theorem~\ref{maintheorem:anabeliantheorem-functorial}, we introduce some auxiliary notation.
Given $K|k$ and $L|l$ as in Theorem~\ref{maintheorem:anabeliantheorem-functorial}, we write
\[ {\Isom_{\Gal}(\Pi_L^a,\Pi_K^a)}_{/\Z_\ell^\times} \]
for the collection of elements such that the induced bijection
\[ \phi_{/\Z_\ell^\times} : {(\Pi_L^a)}_{/\Z_\ell^\times} \cong {(\Pi_K^a)}_{/\Z_\ell^\times} \]
is Galois-equivariant with respect to \emph{some} isomorphism $\eta : \Gal_l \cong \Gal_k$.
Similarly to before, we write
\[ {\Isom^c_{\Gal}(\Pi_L^a,\Pi_K^a)}_{/\Z_\ell^\times} \]
for the intersection of ${\Isom_{\Gal}(\Pi_L^a,\Pi_K^a)}_{/\Z_\ell^\times}$ with ${\Isom^c(\Pi_L^a,\Pi_K^a)}_{/\Z_\ell^\times}$, both of which are considered as subsets of ${\Isom(\Pi_L^a,\Pi_K^a)}_{/\Z_\ell^\times}$.

\begin{remark}
  If $L$ resp. $K$ has positive transcendence degree over its base field, then the action of $\Gal_{l_0}$ resp. $\Gal_{k_0}$ on ${(\Pi_L^a)}_{/\Z_\ell^\times}$ resp. ${(\Pi_K^a)}_{/\Z_\ell^\times}$ is faithful.
  In particular, for any element of ${\Isom_{\Gal}(\Pi_L^a,\Pi_K^a)}_{/\Z_\ell^\times}$ with a corresponding induced bijection
  \[ \phi_{/\Z_\ell^\times} : {(\Pi_L^a)}_{/\Z_\ell^\times} \cong {(\Pi_K^a)}_{/\Z_\ell^\times}, \]
there exists a \emph{unique} isomorphism $\eta : \Gal_{l_0} \cong \Gal_{k_0}$ with respect to which $\phi_{/\Z_\ell^\times}$ is equivariant.
\end{remark}

\begin{proof}[Proof of Theorem~\ref{maintheorem:anabeliantheorem-functorial}]
  This follows by tracing through the proof of Theorem~\ref{maintheorem:anabeliantheorem-reconstruct}, and noting the functoriality at every step.
  We provide the details below.

  By Pop's Theorem (Theorem~\ref{theorem:Pop-recover-functionfields-from-divisorial-inertia}), the map
  \[{\Isom^i(K,L)}_{/\langle \Frob \rangle} \rightarrow {\Isom^c_{\divv}(\Pi_L^a,\Pi_K^a)}_{/\Z_\ell^\times}.  \]
  is a bijection.

  Consider now the subset ${\Isom^i_{\Gal}(K,L)}_{/\langle \Frob \rangle}$ of ${\Isom^i(K,L)}_{/\langle \Frob \rangle}$ consisting of the elements represented by an isomorphism $K^i \cong L^i$ which is equivariant with respect to some isomorphism $\eta : \Gal_l \cong \Gal_k$ of Galois groups.
  Clearly, the image of ${\Isom^i_{\Gal}(K,L)}_{/\langle \Frob \rangle}$ in ${\Isom(\Pi_L^a,\Pi_K^a)}_{/\Z_\ell^\times}$ lands in ${\Isom_{\Gal}(\Pi_L^a,\Pi_K^a)}_{/\Z_\ell^\times}$.
  Arguing as in steps (3) and (7) from the proof of Theorem~\ref{maintheorem:anabeliantheorem-reconstruct} along with Lemma~\ref{lemma:cohom-to-galois-functorial}, we see that ${\Isom^c_{\Gal}(\Pi_L^a,\Pi_K^a)}_{/\Z_\ell^\times}$ is a subset of ${\Isom^c_{\divv}(\Pi_L^a,\Pi_K^a)}_{/\Z_\ell^\times}$.
  In particular, we obtain a canonical bijection:
  \[ {\Isom^i_{\Gal}(K,L)}_{/\langle \Frob \rangle} \rightarrow {\Isom^c_{\Gal}(\Pi_L^a,\Pi_K^a)}_{/\Z_\ell^\times}. \]

  The assertion of Theorem~\ref{maintheorem:anabeliantheorem-functorial} is a straightforward reformulation of these observations.
  Indeed, any isomorphism $\phi$ as in the assertion of Theorem~\ref{maintheorem:anabeliantheorem-functorial}, considered up-to multiplication by $\Z_\ell^\times$, uniquely induces an element of ${\Isom^c_{\Gal}(\Pi_L^a,\Pi_K^a)}_{/\Z_\ell^\times}$ as in steps (4), (5), (6), (7) from the proof of Theorem~\ref{maintheorem:anabeliantheorem-reconstruct}.
  This in turn corresponds to a unique element of ${\Isom^i_{\Gal}(K,L)}_{/\langle \Frob \rangle}$, which is represented by some Galois-equivariant isomorphism $K^i \cong L^i$.
  By the Galois equivariance, we see that this isomorphism restricts to isomorphisms $K_0^i \cong L_0^i$, $k \cong l$ and $k_0 \cong l_0$.
  The remaining assertions of Theorem~\ref{maintheorem:anabeliantheorem-functorial} follow by tracing through the definitions.
\end{proof}

\section{Generic $\ell$-adic Cohomology}\label{section:generic-elladic-cohomology}

Throughout this section, $k$ will be an algebraically closed field and $K$ will denote a function field over $k$.
The main goal for this section is to present a result comparing the so-called \emph{generic $\ell$-adic cohomology} of $K|k$ to the Galois cohomology of $K$.
This comparison result is certainly very well-known to the experts.
As an application, we use the reconstruction results proved above for Galois cohomology to deduce analogous results for generic cohomology.

\subsection{Basics of generic $\ell$-adic cohomology}
We begin by recalling the definition of generic $\ell$-adic cohomology.
Let $\Lambda$ be a quotient of $\Z_\ell$, and let let $X$ be a \emph{model} of $K|k$, i.e. $X$ is an integral $k$-variety with function field $K$.
We define
\[ \H^i(K|k,\Lambda(j)) := \varinjlim_{U} \H^i(U,\Lambda(j)) \]
where the $U$ varies over the Zariski-open $k$-subvarieties of $X$, and $\H^i(U,\Lambda(j))$ is the usual $\ell$-adic cohomology of $U$ with coefficients in $\Lambda(j)$.
Clearly $\H^i(K|k,\Lambda(j))$ does not depend on the choice of $X$.
It is important to note that, in general, $\H^i(K|k,\Lambda(j))$ has \emph{infinite rank}.

Note that the cup-product in $\ell$-adic cohomology induces a notion of a cup-product
\[ \cup : \H^i(K|k,\Lambda(j)) \otimes_{\Lambda} \H^{i'}(K|k,\Lambda(j')) \rightarrow \H^{i+i'}(K|k,\Lambda(j+j')). \]
This cup-product yields a graded-commutative $\Lambda$-algebra structure on
\[ \H^*(K|k,\Lambda(*)) := \bigoplus_{i \geq 0} \H^i(K|k,\Lambda(i)). \]

This construction is functorial with respect to $k$-embeddings of function fields over $k$.
Indeed, if $\iota : K \hookrightarrow L$ is such a $k$-embedding, then we may choose a model $Y$ of $L|k$ and a model $X$ of $K|k$ such that $\iota$ arises from a dominant morphism $f : Y \rightarrow X$.
We thereby obtain a canonical morphism
\[ \iota_* : \H^i(K|k,\Lambda(j)) = \varinjlim_{U} \H^i(U,\Lambda(j)) \xrightarrow{f^*} \varinjlim_U \H^i(f^{-1}(U),\Lambda(j)) \xrightarrow{\text{canon.}} \H^i(L|k,\Lambda(j)). \]
Again, this morphism only depends on the field embedding $\iota$, and it is compatible with cup-products.

Suppose now that $K|k$ is defined over a perfect field $k_0$ by $K_0|k_0$.
Since the open $k$-subvarieties of $X$ which are defined over $k_0$ form a cofinal system among all the open $k$-subvarieties of $X$, it follows that $\H^i(K|k,\Lambda(j))$ inherits a canonical continuous action of $\Gal_{k_0}$, which is compatible with cup-products.
If $L|k$ is defined over $k_0$ by $L_0|k_0$ and $\iota : K \hookrightarrow L$ is a $k$-embedding of function fields which is defined over $k_0$ in the sense that there exists a $k_0$-embedding $\iota_0 : K_0 \hookrightarrow L_0$ whose base-change to $k$ is $\iota$, then the corresponding map
\[ \iota_* : \H^i(K|k,\Lambda(j)) \rightarrow \H^i(L|k,\Lambda(j)) \]
is $\Gal_{k_0}$-equivariant.

\subsection{Comparison with Galois cohomology}\label{subsection:generic-comparison}

Recall that for a model $X$ of $K|k$, the canonical map $\Spec K \rightarrow X$ induces a morphism in $\ell$-adic cohomology
\[ \H^i(X,\Lambda(j)) \rightarrow \H^i(K,\Lambda(j)). \]
Here $\H^i(K,\Lambda(j))$ is just the usual Galois cohomology of $K$, as considered above.
Passing to the colimit over the open $k$-subvarieties of $X$, we obtain a canonical comparison morphism
\[ \H^i(K|k,\Lambda(j)) \rightarrow \H^i(K,\Lambda(j)) \]
between the generic cohomology of $K|k$ and the Galois cohomology of $K$.
This map is compatible with cup-products in the obvious sense.
Furthermore, if $K|k$ is defined over a perfect field $k_0$, then this comparison map is $\Gal_{k_0}$-equivariant.

When using finite coefficients, i.e.\ when $\Lambda = \Z/\ell^n$, it is well known that the comparison map above is an isomorphism.
The situation is different when $\Lambda = \Z_\ell$, in which case $\H^i(K|k,\Lambda(j))$ is generally \emph{not $\ell$-adically complete}, unlike $\H^i(K,\Z_\ell(j))$.
In this case, the Galois cohomology of $K$ can be recovered by $\ell$-adic completion.

\begin{proposition}\label{proposition:generic-to-galois-comparison}
  In the above context, the following hold:
  \begin{enumerate}
  \item One has canonical isomorphisms for all $n \geq 0$:
    \[ \H^i(K|k,\Z_\ell(j)) \otimes_{\Z_\ell} \Z/\ell^n \cong \H^i(K|k,\Z/\ell^n(j)) \cong \H^i(K,\Z/\ell^n(j)). \]
  \item The $\Z_\ell$-modules $\H^i(K|k,\Z_\ell(j))$ are torsion-free for all $i \geq 0$ and all $j$.
  \item The canonical comparison map $\H^i(K|k,\Z_\ell(j)) \rightarrow \H^i(K,\Z_\ell(j))$ becomes an isomorphism after $\ell$-adic completion.
  \end{enumerate}
\end{proposition}
\begin{proof}
  Fix a model $X$ of $K|k$.
  As noted above, the canonical map
  \[ \H^i(K|k,\Z/\ell^n(j)) := \varinjlim_U\H^i(U,\Z/\ell^n(j)) \rightarrow \H^i(K,\Z/\ell^n(j)) \]
  is an isomorphism, where $U$ varies over the open $k$-subvarieties of $X$.
  The exact sequence
  \[ 0 \rightarrow \Z_\ell(j) \xrightarrow{\ell^n} \Z_\ell(j) \rightarrow \Z/\ell^n(j) \rightarrow 0\]
  induces a long exact sequence on $\ell$-adic cohomology~\cite{MR559531}*{Ch. V, Lemma 1.11}.
  Passing to the colimit and using the observation above, we obtain exact sequences of the form
  \[ 0 \rightarrow \H^i(K|k,\Z_\ell(j)) \otimes_{\Z_\ell} \Z/\ell^n \rightarrow \H^i(K,\Z/\ell^n(j)) \rightarrow \H^{i+1}(K|k,\Z_\ell(j))[\ell^n] \rightarrow 0.\]

  Let $f \in K^\times$ be given.
  Then for a sufficiently small open $k$-subvariety $U$ of $X$, $f$ arises from a morphism $f : U \rightarrow \Gbb_m$.
  The image of $f$ in $\H^1(K,\Z/\ell^n(1))$ via the Kummer map agrees with the image of $1 \in \Z_\ell$ under the canonical map
  \[ \Z_\ell = \H^1(\Gbb_m,\Z_\ell(1)) \xrightarrow{f^*} \H^1(U,\Z_\ell(1)) \rightarrow \H^1(K|k,\Z_\ell(1)) \otimes \Z/\ell^n \rightarrow \H^1(K,\Z/\ell^n(1)). \]
  Since $\H^2(\Pbb_k^1 \smin \{0,1,\infty\},\Z_\ell(2)) = 0$, it follows that the norm residue morphism
  \[ \K_i(K) \rightarrow \H^i(K,\Z/\ell^n(i)) \]
  factors through $\H^i(K|k,\Z_\ell(i))$ and hence through $\H^i(K|k,\Z_\ell(i)) \otimes \Z/\ell^n$.
  The Voevodsky-Rost theorem~\cites{MR2811603,MR2597737,Rost1998} then implies that the map
  \[ \H^i(K|k,\Z_\ell(i)) \otimes \Z/\ell^n \rightarrow \H^i(K,\Z/\ell^n(i)) \]
  is surjective, and hence an \emph{isomorphism} by the exactness of the sequence mentioned above.
  The same is then true for
  \[ \H^i(K|k,\Z_\ell(j)) \otimes \Z/\ell^n \rightarrow \H^i(K,\Z/\ell^n(j)) \]
  since $\mu_{\ell^\infty} \subset k$.
  This proves assertion (1) of the proposition, while assertions (2) and (3) follow easily from (1) and the observations made above.
\end{proof}

\subsection{Recovering function fields from their generic $\ell$-adic cohomology}\label{subsection:recovering-from-generic-elladic-cohomology}

Suppose that $K|k$ is defined over a perfect field $k_0$, and that we are given the following data:
\begin{enumerate}
  \item[A.] The profinite group $\Gamma := \Gal(k|k_0)$.
  \item[B.] $\H^1(K|k,\Z_\ell(1))$ and the canonical action of $\Gamma$ on ${\H^1(K|k,\Z_\ell(1))}_{/\Z_\ell^\times}$.
  \item[C.] The set $\Ccal = \{(x,y) \ : \ x,y \in \H^1(K|k,\Z_\ell(1)), \ x \cup y = 0 \in \H^2(K|k,\Z_\ell(2))\}$.
\end{enumerate}
By Proposition~\ref{proposition:generic-to-galois-comparison}, the $\ell$-adic completion of $\H^1(K|k,\Z_\ell(1))$ is canonically isomorphic to the Galois cohomology group $\H^1(K,\Z_\ell(1))$.
Consider the image of $\Ccal$ in ${\H^1(K,\Z_\ell(1))}^2$, denoted by $\Ccal_1$, as well as the set $\Ccal_2$ defined as the closure of the submodule of
\[ \H^1(K,\Z_\ell(1))\hotimes\H^1(K,\Z_\ell(1))  \]
generated by elements of the form $x \otimes y$ for $(x,y) \in \Ccal_1$.
As the Kummer map $K^\times \rightarrow \H^1(K,\Z_\ell(1))$ factors through $\H^1(K|k,\Z_\ell(1))$ (see the proof of Proposition~\ref{proposition:generic-to-galois-comparison}), the Merkurjev-Suslin theorem shows that, for $x,y \in \H^1(K,\Z_\ell(1))$, one has $x \cup y = 0$ if and only if $x \otimes y \in \Ccal_2$.

We therefore see how the data above completely determines (in a visibly functorial way) the given data that appears in Theorem~\ref{maintheorem:anabeliantheorem-reconstruct} resp.~\ref{maintheorem:anabeliantheorem-functorial}.
We thus obtain the following two theorems as corollaries to Theorems~\ref{maintheorem:anabeliantheorem-reconstruct} resp.~\ref{maintheorem:anabeliantheorem-functorial}.

\begin{maintheorem}\label{maintheorem:generic-cohomology-recover}
  Suppose that $k_0$ is a perfect field of characteristic $\neq \ell$ which is not real-closed nor Henselian with respect to any non-trivial valuation.
  Let $K_0$ be a regular function field of transcendence degree $\geq 3$ over $k_0$, and let $\ell$ be a prime which is different from the characteristic of $k_0$.
  Let $k$ denote the algebraic closure of $k_0$ and put $K := K_0 \cdot k$.
  Then the fields $K^i$, $K_0^i$, $k$ and $k_0$, as well as the obvious inclusions among them, are can be reconstructed (uniquely up-to Frobenius twists) from the following data:
  \begin{enumerate}
  \item The absolute Galois group $\Gal(k|k_0)$ of $k_0$, considered as a profinite group.
  \item The $\Z_\ell$-module $\H^1(K|k,\Z_\ell(1))$, and the action of $\Gal(k|k_0)$ on the set ${\H^1(K,\Z_\ell(1))}_{/\Z_\ell^\times}$.
  \item The subset $\{(x,y) \ : \ x,y \in \H^1(K|k,\Z_\ell(1)), \ x \cup y = 0 \}$ of ${\H^1(K|k,\Z_\ell(1))}^2$.
  \end{enumerate}
\end{maintheorem}

\begin{maintheorem}\label{maintheorem:generic-cohomology-functorial}
  In the context of Theorem~\ref{maintheorem:generic-cohomology-recover}, suppose furthermore that $l_0$ is another field which is not real-closed nor Henselian with respect to any non-trivial valuation, and that $L_0|l_0$ is a regular function field of any transcendence degree.
  Let $l$ denote the algebraic closure of $l_0$ and $L := L_0 \cdot l$.
  Suppose that
  \[ \phi : \H^1(K|k,\Z_\ell(1)) \cong \H^1(L|l,\Z_\ell(1)) \]
  is an isomorphism of $\Z_\ell$-modules and $\eta : \Gal(k|k_0) \cong \Gal(l|l_0)$ is an isomorphism of profinite groups, such that the following conditions hold:
  \begin{enumerate}
  \item For all $x,y \in \H^1(K|k,\Z_\ell(1))$, one has $x \cup y = 0$ in $\H^2(K|k,\Z_\ell(2))$ if and only if $\phi(x) \cup \phi(y) = 0$ in $\H^2(L|l,\Z_\ell(2))$.
  \item The induced bijection
    \[ \phi_{/\Z_\ell^\times} : {\H^1(K|k,\Z_\ell(1))}_{/\Z_\ell^\times} \cong {\H^1(L|l,\Z_\ell(1))}_{/\Z_\ell^\times} \]
    is equivariant with respect to the action of $\Gal(k|k_0)$ resp. $\Gal(l|l_0)$ via $\eta$.
  \end{enumerate}
  Then there exists an isomorphism of fields $\psi : K^i \cong L^i$ (unique up-to Frobenius twists) which restricts to an isomorphism $K_0^i \cong L_0^i$, and a unique $\epsilon \in \Z_\ell^\times$, such that $\epsilon \cdot \phi$ is the isomorphism induced by $\psi$, and such that $\eta$ is the isomorphism induced by $\psi$ via the identifications of Galois groups $\Gal(k|k_0) = \Gal(K^i|K_0^i)$ resp. $\Gal(l|l_0) = \Gal(L^i|L_0^i)$.
\end{maintheorem}

\section{Concluding remarks}\label{section:concluding-remarks}

Our proof of Theorems~\ref{maintheorem:generic-cohomology-recover} and~\ref{maintheorem:generic-cohomology-functorial} relied entirely on the fact that we used \emph{integral} $\ell$-adic cohomology, so that we may $\ell$-adically complete to recover Galois cohomology using the comparison from Proposition~\ref{proposition:generic-to-galois-comparison}, and, eventually, construct the pro-$\ell$ abelian-by-central Galois group of the function field in question.
On the other hand, while there is a canonical Galois-equivariant comparison morphism
\[ \H^i(K|k,\Q_\ell(j)) \rightarrow \H^i(K,\Q_\ell(j)) \]
there is no apparent way of recovering $\H^i(K,\Q_\ell(j))$ from $\H^i(K|k,\Q_\ell(j))$ alone.

In fact, it is currently unknown whether a higher-dimensional function field $K|k$, which is defined over, say, a finitely-generated field $k_0$, can be recovered from the generic cohomology ring $\H^*(K|k,\Q_\ell(*))$ or the Galois-cohomology ring $\H^*(K,\Q_\ell(*))$ endowed with the action of $\Gal(k|k_0)$.
We expect this to indeed be the case.
The main difficulty in proving such a result seems to be in the \emph{local theory}, as there is currently no known way to detect valuations of $K|k$ when using $\Q_\ell$-coefficients.
On the other hand, there is a local theory when one uses $\Q$-coefficients, and this was the basis of the (rational) Hodge-theoretic variant of Theorem~\ref{maintheorem:generic-cohomology-recover} proven in~\cite{TopazTorelli}.
In more general terms, it is an interesting to ask whether anabelian results can be obtained from cohomology \emph{independently} of the choice coefficients.

\bibliographystyle{amsalpha}
\bibliography{refs.bib}
\end{document}